\documentclass[runningheads, envcountsame, a4paper]{llncs}

\usepackage{amssymb}
\usepackage{amssymb}
\usepackage{graphicx}
\usepackage{enumerate}
\usepackage{amsmath,amsopn,amsfonts}
\usepackage[hyphens]{url}
\usepackage{microtype}

\usepackage{url}
\urldef{\mailsa}\path| berdinsky@gmail.com, 
p.trakuldit@gmail.com |

\newcommand{\keywords}[1]{\par\addvspace\baselineskip
\noindent\keywordname\enspace\ignorespaces#1}

\begin{document}

\mainmatter  

\title{Measuring Closeness between Cayley 
Automatic Groups and Automatic Groups}

\titlerunning{Measuring closeness between Cayley 
automatic groups and automatic groups}

%
%
\author{Dmitry Berdinsky\inst{1,2}
 \and 
        Phongpitak Trakuldit\inst{1,2} 
} 

%
\authorrunning{D. Berdinsky and 
               P. Trakuldit}


\institute{Department of Mathematics, Faculty of Science, 
  Mahidol University,  
 Bangkok, Thailand  \and 
 Centre of Excellence in Mathematics, 
 Commission on Higher Education, Bangkok, Thailand  
 \\
 \mailsa\\
}

%
%

\toctitle{Measuring Closeness between Cayley 
Automatic Groups and Automatic Groups}
\tocauthor{Dmitry~Berdinsky and Phongpitak~Trakuldit}
\maketitle
\setcounter{footnote}{0}

\begin{abstract}      

 In this paper we introduce 
 a way to estimate a level of 
 closeness of Cayley automatic 
 groups to the class of automatic groups
 using a certain 
 numerical characteristic.  
 We characterize Cayley automatic 
 groups which are not automatic
 in terms of this numerical characteristic 
 and then study it for 
 the lamplighter group, the Baumslag--Solitar groups 
 and the Heisenberg group.



\end{abstract}

 \section{Introduction}

    Cayley automatic groups had been introduced 
    by Kharlampovich, Khoussainov and Miasnikov 
    as a generalization of automatic groups \cite{KKM11}.
    They are all finitely generated groups 
    for which their directed labeled Cayley graphs are
    finite automata presentable structures
    (automatic structures) 
\cite{khoussainov2007three,KhoussainovNerode95,khoussainov2008open,BakhFrankSasha2007}; 
    see, e.g., also  
    \cite{Blumensath99,BlumensathGradel00,NiesThomas08,OliverThomas05,Pelecq97,RubinThesis,Senizergues92}.  
    In particular, Cayley automatic groups include 
    all automatic groups 
    in the sense of Thurston and others \cite{Epsteinbook}.      
    Cayley automatic groups inherit the key 
    algorithmic properties of automatic groups:     
    the first order theory 
    for a directed labeled Cayley graph of 
    a Cayley automatic groups is decidable, 
    the word problem in a Cayley automatic group 
    is decidable in quadratic time~\cite{KKM11}. 
    The set of Cayley automatic groups comprise 
    all finitely generated nilpotent groups 
    of nilpotency class at most two \cite{KKM11},
    the Baumslag--Solitar groups \cite{dlt14}, 
    higher rank lamplighter groups \cite{Taback18} 
    and all fundamental groups of $3$--dimensional 
    manifolds.   
    This shows that Cayley automatic groups 
    include  important classes of 
    groups. 
        
    In this paper we introduce the classes 
    of Cayley automatic groups $\mathcal{B}_f$ defined
    by non--decreasing and non--negative functions 
    $f$.
    Informally speaking, for any given group 
    $G \in \mathcal{B}_f$, 
    the function $f$ shows an upper bound  
    for a level of closeness of the group $G$ 
    to the class of automatic groups. 
    In particular, if $f$ is identically 
    equal to zero, then $G$ must be automatic.     
    So, similarly to a growth function, 
    one can consider $f$ as a numerical characteristic 
    of the group $G$. 
    Studying numerical characteristics of groups 
    and relations between them is an important topic 
    in group theory~\cite{Vershik99}.   
    In this paper we initiate study of this 
    numerical characteristic. We first characterize 
    non--automatic groups in terms of this 
    characteristic. Then we study this 
    characteristic for some
    non--automatic groups, namely, 
    the lamplighter group $\mathbb{Z}_2 \wr \mathbb{Z}$, 
    the Baumslag--Solitar groups $BS(p,q)$, with 
    $1 \leqslant p < q$, and the Heisenberg group 
    $\mathcal{H}_3 (\mathbb{Z})$.    
    Another motivation to introduce this 
    numerical characteristic is to address the
    problem of finding characterization for 
    Cayley automatic groups by studying classes 
    $\mathcal{B}_f$ for some functions $f$.  
    
    The paper is organized as follows.
    In Section \ref{secpreliminaries} we 
    recall the definitions of automatic and Cayley 
    automatic groups. 
    Then we give the definition of the classes 
    of Cayley automatic groups $\mathcal{B}_f$  
    and show that it does not 
    depend on the choice of generators.    
    In Section \ref{seccharnonautomaticgroups} 
    we give a characterization of non--automatic
    groups by showing that if $G \in \mathcal{B}_f$ 
    is non--automatic, 
    then $f$ must be unbounded. 
    In Sections 
    \ref{secbaumslagsolitar}    
    and    
    \ref{seclamplighter}  
    we show that   
    the Baumslag--Solitar groups 
    $BS(p,q)$, with $1 \leqslant p  < q$, 
    and the lamplighter group 
    $\mathbb{Z}_2 \wr \mathbb{Z}$
    are in the class $\mathcal{B}_{\mathfrak{i}}$, 
    where $\mathfrak{i}$ is the identity function: 
    $\mathfrak{i}(n)=n$. 
    Moreover, we show that these  groups 
    cannot be elements of any class 
    $\mathcal{B}_f$, if the function $f$ 
    is less than 
    $\mathfrak{i}$ in coarse sense 
    (see Definition \ref{coarseineqdef}).         
    In Section \ref{secheisenberggroup} we show 
    that the Heisenberg group 
    $\mathcal{H}_3(\mathbb{Z})$ is  
    in the class $\mathcal{B}_\mathfrak{e}$, 
    where $\mathfrak{e}$ is the exponential 
    function: $\mathfrak{e}(n)=\exp(n)$. 
    We then show that $\mathcal{H}_3(\mathbb{Z})$ 
    cannot be an element of any class 
    $\mathcal{B}_f$, if $f$ is less
    than the cubic root function $\sqrt[3]{n}$ in
    coarse sense. Section \ref{secdiscussion} 
    concludes the paper.

 \section{Preliminaries} 
 
 \label{secpreliminaries} 
 
 Let $G$ be a finitely generated infinite group.    
 Let $A \subseteq G$ be a finite generating set 
 of the group $G$.  
 We denote by $S$ the set $S = A \cup A^{-1}$, where 
 $A^{-1}$ is the set of the inverses of elements of $A$.  
 For given elements $g_1, g_1 \in G$, 
 we denote by $d_A (g_1, g_2)$ the distance 
 between the elements $g_1$ and $g_2$ in the Cayley graph 
 $\Gamma (G,A)$. Similarly, we denote by 
 $d_A (g)= d_A (e,g)$ the word length of $g$ with respect
 the generating set $A$.    
 We denote by $\pi: S^* \rightarrow G$ 
 the canonical mapping which sends every 
 word $w \in S^*$ to the corresponding 
 group element $\pi(w)=\overline{w} \in G$.      

 We assume that the reader is familiar 
 with the notion of finite automata and 
 regular languages.
 For a given finite alphabet $\Sigma$   
 we put $\Sigma_\diamond = \Sigma \cup \{\diamond \}$, 
 where 
  $\diamond \notin \Sigma$ is a padding 
  symbol. The convolution of 
  $n$ words $w_1,\dots, w_n \ \in \Sigma^*$ is the string 
  $w_1 \otimes \dots \otimes w_n$		      
  of length $\max\{|w_1|,\dots,|w_n|\}$ over the alphabet 
  $\Sigma_\diamond ^n$ defined as follows. 
  The $k$th symbol
  of the string is $(\sigma_1,\dots,\sigma_n )^\top$, 
  where $\sigma_i$, $i=1,\dots,n$ is the $k$th symbol of 
  $w_i$ if $k \leqslant |w_i|$ and
  $\diamond$ otherwise. 
  The convolution $\otimes R$  of a $n$--ary relation 
  $R \subseteq \Sigma^{*n} $ is defined as
  $\otimes R = \{w_1 \otimes \dots \otimes w_n | 
  (w_1, \dots, w_n) \in R\}$.
  We recall that a $n$--tape synchronous 
  finite automaton is
  a finite automaton over the alphabet 
  $\Sigma_\diamond  ^n \setminus 
  \{(\diamond,\dots,\diamond)\}$. 
  We say that a $n$--ary relation 
  $R \subseteq \Sigma^{*n}$ is regular if  
  $\otimes R$ is accepted by a $n$--tape synchronous 
  finite automaton. 

 Below we give a definition of 
 automatic groups in the sense of Thurston and others 
 \cite{Epsteinbook}. 
 
 \begin{definition}
 \label{defautomatic}        
    We say that $G$ is automatic if 
    there exists a regular language 
    $L \subseteq S^{*}$ such that 
    $\varphi =  \pi|_L : L \rightarrow G$ 
    is a bijection and for every 
    $a \in A$ the binary relation 
    $R_a=
     \{(\varphi^{-1}(g),\varphi^{-1}(ga)) \, | \, g\in G\} 
     \subseteq L \times L$
    is regular.    
 \end{definition}
 \begin{remark}   
  In Definition \ref{defautomatic} we required 
  $\varphi$ to be bijective while in 
  the original definition 
  \cite[Definition~2.3.1]{Epsteinbook} 
  $\varphi  = \pi |_L : L \rightarrow G$ 
  is surjective but it is additionally required 
  that the equality relation 
  $R_e= \{(u,v) \in L \times L \, | \, 
   \pi(u) = \pi(v) \}$ is regular. 
  It can be seen that these two definitions are equivalent.  
  Clearly, if a group $G$ is automatic in the sense of 
  Definition \ref{defautomatic}, 
  then it is automatic in the sense
  of \cite[Definition~2.3.1]{Epsteinbook}. 
  Now, suppose that $G$ is automatic in 
  the sense of \cite[Definition~2.3.1]{Epsteinbook}. 
  Then there is a regular language 
  $L \subseteq S^*$ for which 
  the map $\varphi = \pi |_L : L \rightarrow G$ 
  is surjective and   
  the relations $R_a, a \in A$ 
  and $R_e$ are regular.   
  Let 
  $L' = \{ w \in L \, | \, 
  \left(\forall u <_{llex} w \right) \pi(u) \neq \pi(w)\}$, 
  where $llex$ is a length--lexicographical order. 
  Then $\varphi' = \pi|_{L'} : L' \rightarrow G$ 
  is bijective and for every $a \in A$ the binary relation 
  $R_a'= \{\left(\varphi'^{-1} (g), 
  \varphi'^{-1}(ga) \right) \, | \, g\in G \}
  \subseteq L' \times L'$ 
  is regular. That is, $G$ is automatic in the sense of Definition \ref{defautomatic}. 
 \end{remark}
 \begin{definition}
 \label{defcayleyautomatic}    
    We say that $G$ is Cayley automatic if 
    there exist a regular language 
    $L \subseteq S^*$ and a bijection 
    $\psi: L \rightarrow G$ such that 
    for every $a \in A$ the binary relation 
    $R_a = \{(\psi^{-1}(g),\psi^{-1}(ga))| 
    g \in G\} \subseteq L \times L$ is regular. 
    We call $\psi: L \rightarrow G$ 
    a Cayley automatic representation of $G$.       
 \end{definition}     
%
    
    We denote by $\mathcal{A}$ and $\mathcal{C}$ 
    the classes of all automatic and Cayley automatic groups, respectively.       
    Clearly, $\mathcal{A} \subseteq \mathcal{C}$.  
   However, $\mathcal{A}$ is a proper subset of 
   $\mathcal{C}$:
   for example, the lamplighter 
   group, the Baumslag--Solitar groups and
   the Heisenberg group 
   $\mathcal{H}_3 (\mathbb{Z})$
   are Cayley automatic, 
   but not automatic. 
   We will refer to $\mathbb{N}$ as the set of 
   all positive integers.  
   We denote by $\mathbb{R}^{+}$ the set of 
   all non--negative real numbers.  
   Let $\mathfrak{F}$ be the 
   following set of non--decreasing 
   functions: 
   $$\mathfrak{F} = \{ f: [Q, + \infty ) 
   \rightarrow \mathbb{R}^{+}  \vert
   [Q, + \infty ) \subseteq \mathbb{N} \wedge 
   \forall n(n \in \mathrm{dom}\,f 
   \implies f(n)\leqslant f(n+1))\}.$$  
   \begin{definition}
   \label{coarseineqdef}   
    Let $f,h \in \mathfrak{F}$.    
    We say that $h \preceq f$ if 
    there exist positive integers  
    $K,M$ and $N$    
    such that
    $[N,+\infty) \subseteq 
    \mathrm{dom}\,h \cap \mathrm{dom}\,f$ and     
    $h (n) \leqslant K f(M n)$
    for every integer 
    $n \geqslant N$.    
    We say that 
    $h \asymp f$ if $h \preceq f$ and 
    $f \preceq h$. 
    We say that $h \prec f$ if 
    $h \preceq f$ and $h \not\asymp f$. 
 \end{definition}
  Let $G \in \mathcal{C}$ be a Cayley automatic group
  and $f \in \mathfrak{F}$. 
  Let us choose some finite generating set 
  $A \subseteq G$.     
  For a given language $L \subseteq S^*$ 
  and $n \in \mathbb{N}$  
  we denote by 
  $L^{\leqslant n}$ the set of all words 
  of length less than or equal to $n$ 
  from  the language $L$, i.e.,   
  $L^{\leqslant n} =\{w \in L \, | \, |w| \leqslant n\}$.
  \begin{definition} 
  \label{maindef1}    
    We say that $G \in \mathcal{B}_f$ if 
    there exist a regular language 
    $L \subseteq S^*$ and a Cayley automatic
    representation 
    $\psi : L \rightarrow G$ such 
    that for the function 
    $h \in \mathfrak{F}$,  
    defined by the equation 
    \begin{equation} 
    \label{maineq1}       
      h(n) = 
      \max \{d_A (\pi(w),\psi(w))| 
      w \in L^{\leqslant n}\},
    \end{equation} 
    the inequality $h \preceq f$ holds.   
  \end{definition}    
  We denote by $\mathcal{B}_f$ the class of
  all Cayley automatic groups $G$ for which 
  $G \in \mathcal{B}_f$.  
  Proposition \ref{genindprop1} below shows  
  that Definition \ref{maindef1} does not 
  depend on the choice of generating set $A$.      
  \begin{proposition}
  \label{genindprop1}              
     Definition \ref{maindef1} does not depend on 
     the choice of generating set.  
  \end{proposition}
  \begin{proof}          
     Let $A' \subseteq G$ be another generating 
     set of $G \in \mathcal{B}_f$. 
     We put $S' = A' \cup A'^{-1}$. 
     In order to simplify an exposition of our proof,
     we will assume that $e \in A'$. 
     Let us represent every element $g \in S$ 
     by a word $w_g \in S'^{*}$, i.e., 
     $\pi (w_g) = g$, for which the lengths 
     of the words $|w_g|$ are the same for all 
     $g \in S$. In order to make the 
     lengths $w_g, g \in S$ equal, one can use 
     $e \in  S'$ as a padding symbol.     
     Let us canonically extend the mapping  
     $g \mapsto w_g, g\in S$  
     to the monoid homomorphism 
     $\xi: S^* \rightarrow S'^*$.

     We remark that the definition of $\xi$ 
     ensures that $\pi(\xi(w))=\pi(w)$ for $w \in S^*$.    
     For a given Cayley automatic representation 
     $\psi: L \rightarrow G$ for which 
     $h \preceq f$, we construct a new 
     Cayley automatic representation 
     $\psi': L' \rightarrow G$ as follows. 
     We put $L' = \xi (L) \subseteq S'^*$ and  
     define a bijection        
     $\psi' : L' \rightarrow G$ as
     $\psi' = \psi \circ \tau$, where 
     $\tau = (\xi|_L)^{-1}$.  
     It can be seen that $\psi'$ is a Cayley 
     automatic representation of $G$. Furthermore, 
     for the function $h' \in \mathfrak{F}$ defined 
     by \eqref{maineq1} 
     with respect to $\psi'$ we obtain that 
     $h ' \preceq h$ which implies that $h' \preceq f$.        
     This proof can be generalized for the case when 
     $e \notin A'$.   \qed
  \end{proof}
  We denote by ${\bf z} \in \mathfrak{F}$ the  
  zero function: ${\bf z}(n)=0$ for all 
  $n \in \mathbb{N}$. 
  By Definition \ref{maindef1}, 
  we have that $\mathcal{B}_{\bf z} = \mathcal{A}$. 
  Proposition \ref{increl1} below shows some elementary 
  properties of the classes $\mathcal{B}_f$.    
  \begin{proposition}  
  \label{increl1}   
    If $f \preceq g$, then 
    $\mathcal{A} \subseteq \mathcal{B}_f 
       \subseteq \mathcal{B}_g \subseteq \mathcal{C}$. 
       If $f \asymp g$, then  
       $\mathcal{B}_f = \mathcal{B}_g$.  
  \end{proposition}  
  \begin{proof}
    By definition, every group of the class 
    $\mathcal{B}_g$ 
    is Cayley automatic, i.e., 
    $\mathcal{B}_g \subseteq \mathcal{C}$. 
    The inclusion 
    $\mathcal{A} \subseteq \mathcal{B}_f$  
    follows from the fact that 
    ${\bf z} \preceq f$ for every $f \in \mathfrak{F}$. 
    The transitivity of the relation $\preceq$ on 
    $\mathfrak{F}$ implies that if $f \preceq g$, then   
    $\mathcal{B}_f \subseteq \mathcal{B}_g$.  
    The fact that $f \asymp g$ implies 
    $\mathcal{B}_f = \mathcal{B}_g$ 
    is straightforward. \qed
  \end{proof}

 \section{Characterizing Non--Automatic Groups}  
 \label{seccharnonautomaticgroups}   
   
  Let $G$ be a Cayley automatic group, 
  $A \subseteq G$ be a finite generating set 
  and $S = A \cup A^{-1}$. 
  Given a word $w \in S^*$, for a 
  non--negative integer $t$ we put 
  $w(t)$ to be the prefix of $w$ of a length $t$, 
  if $t \leqslant |w|$, and $w(t)=w$, if $t>|w|$.
  Following the notation from \cite{Epsteinbook},
  we denote by 
  $\widehat{w}: [0,\infty) 
  \rightarrow \Gamma (G,A)$ 
  the corresponding path in 
  the Cayley graph $\Gamma(G,A)$
  defined as follows. If $t \geqslant 0$ is an integer, 
  then $\widehat{w}(t)= \pi (w(t))$, and
  $\widehat{w}$ is extended to non--integer values  
  of $t$ by moving along the respective edges 
  with unit speed.  
  Given words $w_1,w_2 \in S^*$ and 
  a constant $C_0 \geqslant 0$, we say 
  that the paths $\widehat{w_1}$ and 
  $\widehat{w_2}$ are a uniform distance
  less than or equal to $C_0$ apart if 
  $d_A (\widehat{w_1}(t),\widehat{w_2}(t))
   \leqslant C_0$ for all non--negative integers $t$.     
    
  Theorem \ref{chartheoremepstein1} 
  below is a simplified modification of the 
  theorem characterizing automatic 
  groups due to Epstein 
  et al.~\cite[Theorem~2.3.5]{Epsteinbook}. 
  This theorem follows from the existence 
  of standard 
  automata~\cite[Definition~2.3.3]{Epsteinbook} 
  for all elements of $A$.  
  For the existence of standard automata it is enough
  to assume the solvability of the word problem in $G$. 
  We recall that for Cayley automatic 
  the word problem in $G$ is 
  decidable~\cite[Theorem~8.1]{KKM11}.   
  \begin{theorem}(\cite[Theorem~2.3.5]{Epsteinbook})
  \label{chartheoremepstein1}     
     Let $L \subseteq S^*$ be a regular language 
     such that $\pi : L \rightarrow G$ is surjective. 
     Assume that there is a constant $C_0$ such that
     for every $w_1,w_2 \in L$ and 
     $a \in A$ for which $\pi(w_1)a = \pi(w_2)$,       
     the paths $\widehat{w_1}$ and $\widehat{w_2}$ 
     are a uniform distance less than
     or equal to $C_0$ apart.
     Then $G$ is an automatic group.  
  \end{theorem}       
 
  Let $d \in \mathfrak{F}$ be any bounded function 
  which is not identically equal to the zero 
  function ${\bf z}$.  Although 
  ${\bf z} \prec d$, the theorem below 
  shows that the class $\mathcal{B}_d$ 
  does not contain any non--automatic group.  
   
  \begin{theorem} 
  \label{chartheorem1}           
    The class $\mathcal{B}_d = \mathcal{A}$. 
    In particular, if 
    for any function $f \in \mathfrak{F}$    
    the class $\mathcal{B}_f$   
    contains a non--automatic group, then 
    $f$ must be unbounded.       
  \end{theorem}
  \begin{proof}
     Let us show that 
     $\mathcal{B}_d = \mathcal{A}$. 
     By Proposition \ref{increl1}, we only need to 
     show that $\mathcal{B}_d 
     \subseteq \mathcal{A}$. 
     Assume that $G \in \mathcal{B}_d$. 
     By Definition \ref{maindef1}, 
     there exists a Cayley automatic
     representation  $\psi_0:L_0 \rightarrow G$ 
     for some $L_0 \subseteq S^*$ 
     such that, for the function 
     $h_0 (n) = 
     \max \{d_A (\pi(w),\psi_0(w)) | 
     w \in L_0^{\leqslant n}\}$, 
     $h_0 \preceq d$.  
     This implies that $d_A (\psi_0 (w), \pi (w))$
     is bounded from above by some constant $K_0$  
     for all $w \in L_0$. 
     
     We put 
     $L_1 = S^{* \leqslant K_0}$.     
     Let $L = L_0 L_1$ 
     be the concatenation of $L_0$ and $L_1$.     
     The language $L$ is regular.   
     For any given $g \in G$, 
     $d_A (\pi(\psi_0^{-1}(g)),g) \leqslant K_0$. 
     This implies that there is a word 
     $u \in L_1$ such that, for the concatenation 
     $w = \psi_0^{-1}(g)u$, $\pi(w) = g$.        
     Therefore, the map $\pi: L \rightarrow G$ is 
     surjective. 
     
     Let $w_1,w_2 \in L$ be some words 
     for which $\pi(w_1)a= \pi (w_2)$, $a \in A$.
     There exist words $v_1,v_2 \in L_0$ and
     $u_1,u_2 \in L_1$ for which $w_1 =v_1 u_1$ 
     and $w_2 = v_2 u_2$. We obtain 
     that $d_A (\psi_0(v_1),\psi_0(v_2)) 
     \leqslant 
     d_A (\pi_0(v_1),\pi_0(v_2))+ 2K_0 
     \leqslant d_A (\pi (w_1), \pi (w_2)) 
     + 2 K_0 + 2K_0 \leqslant 4K_0+1$. That is, 
     there exists  
     $g \in G$, for which 
     $d_A (g) \leqslant 4K_0 +1$, such that 
     $\psi_0 (v_1)g = \psi_0 (v_2)$. 
     The pair $(v_1, v_2)$ is accepted 
     by some two--tape synchronous automaton $M_g$. 
   
     Let $N_g$ be the number of states of $M_g$.    
     Given a non--negative integer 
     $t$, there exist words $p_1,p_2 \in S^*$, 
     for which the lengths $|p_1|,|p_2|$ are 
     bounded from above by 
     $N_g$, such that 
     the pair $(v_1(t) p_1, v_2(t) p_2)$ is accepted 
     by $M_g$; in particular, 
     $v_1(t) p_1, v_2(t) p_2 \in L_0$. We obtain that:
     \begin{equation*}
     \begin{split}
      d_A (\pi (v_1 (t)),\pi(v_2(t))) \leqslant 
      d_A (\pi (v_1 (t)p_1), \pi(v_2 (t)p_2)) + 
      |p_1| + |p_2| \leqslant  \\
      d_A (\psi_0 (v_1 (t)p_1), \psi_0 (v_2 (t)p_2)) 
      + 2 K_0 + 2 N_g  \leqslant \\  d_A(g) + 2 K_0 + 2 N_g  
      \leqslant 6 K_0 + 2 N_g + 1.
      \end{split}
      \end{equation*} 
      Therefore, 
      \begin{equation*}
      \begin{split} 
       d_A (\widehat{w_1}(t), \widehat{w_2}(t))=
       d_A (\pi(w_1 (t)), \pi(w_2 (t))) \leqslant  
       d_A (\pi(v_1 (t)), \pi (v_2 (t))) + 2 K_0 
       \leqslant \\ 8 K_0 + 2 N_g + 1.
       \end{split}
       \end{equation*}  
       There are only finitely many 
       $g$ for which $d_A (g) \leqslant 4K_0 +1$, 
       so $N_g$ can be bound by some constant $N_0$. 
       Thus, for  $C_0=8 K_0 + 2 N_0 + 1$,
       we obtain that 
       $d_A (\widehat{w_1}(t), \widehat{w_2}(t)) 
       \leqslant C_0$, that is, the paths 
       $\widehat{w_1}$ and $\widehat{w_2}$ are a uniform
       distance $C_0$ apart.  By Theorem
       \ref{chartheoremepstein1}, the group  
       $G$ is automatic. 
       The second statement of the theorem is 
       straightforward.  \qed   
  \end{proof}

 \section{The Baumslag--Solitar Groups} 
 \label{secbaumslagsolitar} 
   
   Let us consider 
   the Baumslag--Solitar groups
   $BS(p,q) = \langle a, t  | t a^p t^{-1} = a^q \rangle$ 
   with $1 \leqslant p < q$. 
   These groups are not automatic,  
   see Epstein et al.~\cite[Section~7.4]{Epsteinbook}, 
   but they are Cayley automatic \cite[Theorem~3]{dlt14}. 
   The Cayley automatic  
   representations of the Baumslag--Solitar  
   groups constructed  in \cite[Theorem~3]{dlt14}
   use the normal form 
   obtained from representing 
   these groups as 
   the HNN extensions  \cite[Corollary~2]{dlt14}.  
   This normal form is shown in
   the following proposition.  
  \begin{proposition}  
  \label{BSnormalformlemma1}     
     Any element $g \in BS(p,q)$ for 
     $1 \leqslant p < q$ can be written 
     uniquely as $g= \widetilde{w}(a,t)a^k$, where           
     $$
      \widetilde{w}(a,t) \in \{ t, at, \dots, a^{q-1}t, 
      t^{-1}, at^{-1}, \dots, a^{p-1}t^{-1}\}^*
     $$ 
     is freely reduced and $k \in \mathbb{Z}$. 
  \end{proposition}  
  Let us now describe a modification of 
  the Cayley automatic representation of 
  $BS(p,q)$ constructed 
  in \cite[Theorem~3.2]{dlt14} which 
  is compatible with Definition \ref{defcayleyautomatic}.
  We put $a_1 = a, \dots, a_{q-1} = a^{q-1}$.  
  Let $A = \{a_0,a_1,\dots,a_{q-1},t\}$ and 
  $S = A \cup A^{-1} = 
   \{e,a_1,a_2,\dots,a_{q-1},
     a_1^{-1},\dots,a_{q-1}^{-1},t,t^{-1}\}$.
  Given an element 
  $g = \widetilde{w}(a,t)a^k\in BS(p,q)$, 
  we construct 
  the word $w=uv$ which is the concatenation 
  of two words $u,v \in S^*$ defined as follows. 
 
  The word $u \in \{t,t^{-1}, a_1,\dots,a_{q-1} \}^*$ 
  is obtained from the corresponding word 
  $\widetilde{w}(a,t)$   
  by changing the subwords 
  $at^\epsilon, \dots, a^{q-1}t^{\epsilon}$ 
  to the subwords 
  $a_1t^\epsilon, \dots, a_{q-1}t^{\epsilon}$,
  respectively,  
  where $\epsilon=+1$ or $\epsilon=-1$.      
  The word $v$ is obtained from the $q$--ary
  representation of $|k|$ by changing the 
  $0$ to $e$ and 
  $1,\dots,q-1$ to 
  $a_1,\dots,a_{q-1}$ and 
  $a_1^{-1},\dots,a_{q-1}^{-1}$, if 
  $k\geqslant 0$ and $k<0$, respectively.  
  The set of all such words $w$ is a regular language 
  $L \subseteq S^*$. 
  Thus, we have constructed a bijection 
  $\psi: L \rightarrow BS(p,q)$. 
  
  By \cite[Theorem~3.2]{dlt14}, $\psi$ provides 
  a Cayley automatic representation of $BS(p,q)$. 
  It is worth noting that if $g \in BS(p,q)$ is 
  an element for which $k=0$, then 
  for $w= \psi^{-1}(g)$ we obtain that $\psi(w)=\pi(w)$.   
  Let $\widetilde{A}=\{a,t\}$. We have the 
  following metric  estimates for 
  the groups $BS(p,q)$. 
  \begin{theorem}(\cite[Theorem~3.2]{BurilloElder14})
  \label{BSmetrictheorem1}     
     There exist constants $C_1,C_2,D_1,D_2>0$
     such that for every element $g \in BS(p,q)$
     for $1 \leqslant p < q$ written as 
     $\widetilde{w}(a,t)a^k$, we have: 
     $   C_1 (|\widetilde{w}|+\log (|k|+1)) - D_1 
        \leqslant d_{\widetilde{A}} (g)   \leqslant
        C_2 (|\widetilde{w}|+\log (|k|+1)) + D_2      
     $.
  \end{theorem}
  \begin{remark}
  	 The normal form 
  	 for the elements of $BS(p,q)$
  	 used in \cite{BurilloElder14}
  	 is almost the same as in Proposition 
  	 \ref{BSnormalformlemma1} modulo 
  	 the choice for the range of powers of $a$. 
  	 Namely, any element $g \in BS(p,q)$ for
  	 $1 \leqslant p < q$ can also be 
  	 written uniquely as $g = w(a,t)a^k$ such that 
  	 \begin{equation*}
  	 \begin{split}  
  	   w(a,t) \in  \{t,at,a^2 t,\dots,a^\alpha t, 
  	   a^{-1}t, a^{-2}t,\dots,a^{-\beta}t, 
  	   t^{-1}, at^{-1}, a^2 t^{-1},\dots, \\ 
  	   a^{\gamma}t^{-1}, 
  	   t^{-1}, a^{-1}t^{-1}, 
  	   a^{-2}t^{-1},\dots,
  	   a^{-\delta}t^{-1}\}^*
  	\end{split} 
  	\end{equation*}
    is freely reduced, 
    where 
    $\alpha = \lfloor \frac{q}{2} \rfloor, 
     \beta = \lfloor \frac{q-1}{2}\rfloor, 
     \gamma = \lfloor \frac{p}{2} \rfloor$ and  
     $\delta = \lfloor \frac{p-1}{2} \rfloor$, 
     see \cite[Lemma~3.1]{BurilloElder14}. 
    It can be seen that the metric estimates
    obtained in \cite{BurilloElder14} for 
    the normal form $w(a,t)a^k$ remain
    valid for the normal form $\widetilde{w}(a,t)a^k$ 
    modulo changing the constants $C_1$ and $C_2$. 
  \end{remark}	
  It follows from Theorem \ref{BSmetrictheorem1} 
  that there exist constants  $C_1',C_2',D_1',D_2'>0$
  such that for every element $g \in BS(p,q)$ and for 
  the corresponding word $\psi^{-1}(g) = uv$ we have
  \begin{equation}
  \label{BSmetriceq1}    
     C_1' (|u|+ |v|) - D_1' 
     \leqslant d_{A}(g) \leqslant 
     C_2' (|u|+ |v|) + D_2'.
  \end{equation}
 
 \begin{theorem} 
 \label{BStheorem1}   
    Given $p$ and $q$ 
    for which $1 \leqslant p < q$, 
    the Baumslag--Solitar group     
    $BS(p,q) \in \mathcal{B}_{\mathfrak{i}}$.
    Moreover, for any $f \prec \mathfrak{i}$, 
    $BS(p,q) \notin \mathcal{B}_f$.  
 \end{theorem} 
 \begin{proof}
    For given $p$ and $q$ for which $1 \leqslant p < q$
    let us consider the Cayley automatic representation
    $\psi: L \rightarrow BS(p,q)$ constructed above.     
    Let $h$ be the function given by \eqref{maineq1} 
    with respect to this Cayley automatic representation.    
    We will show that $h \preceq \mathfrak{i}$
    (in fact one can verify that 
    $h \asymp \mathfrak{i}$).      
    Let $w = uv \in L ^{\leqslant n}$ and 
    $g = \psi (w)$ be the corresponding 
    group element of $BS(p,q)$. By \eqref{BSmetriceq1}, 
    there exists a constant $C$ such that
    $d_{A} (g) \leqslant C(|u| + |v|) = C |w|$.
    Therefore, 
    $d_A (\pi(w),\psi(w)) 
    \leqslant n + d_A (g) \leqslant (C+1) n$. 
    Therefore, $h \preceq \mathfrak{i}$ 
    which implies that  
    $BS(p,q) \in \mathcal{B}_{\mathfrak{i}}$.      
    
    Let us show now the second statement of 
    the theorem.     
    Suppose that $BS(p,q) \in \mathcal{B}_f$ for 
    some $f \prec \mathfrak{i}$.        
    Then there exists a  Cayley automatic 
    representation $\psi': L' \rightarrow BS(p,q)$ 
    for which $h' \preceq f$, where 
    $h'$ is given by \eqref{maineq1}. 
    We have 
    $h' \prec \mathfrak{i}$. 
    We recall that for a group 
    $\langle   X | R \rangle$ given by 
    a set of generators $X$ and a set
    of relators $R$ the Dehn function      
    is given by     
    $D(n)= \max_{u \in U_n} 
     \{ \mathrm{area}(u)\}$, where 
    $U_n = \{u \in (X \cup X^{-1})^* | 
            \pi(u)=e \wedge |u| \leqslant n\}$ 
    is the set of words of the length at most $n$ 
    representing the identity of the group 
    $\langle X | R \rangle$  and  
    $\mathrm{area}(u)$ is the combinatorial area of 
    $u$  which 
    is the minimal $k$ for which 
    $u = \prod_{i=1}^{k} v_i r_i^{\pm 1} v_i^{-1}$ 
    in the free group $F(X)$, where $r_i \in R$.       

    Let $w \in \{a,a^{-1},t,t^{-1}\}^{*}$ be 
    a word representing the identity in $BS(p,q)$ 
    for which $|w| \leqslant n$. 
    The word $w$ corresponds to a loop in the Cayley 
    graph $BS(p,q)$ with respect 
    to the generators $a,t$. 
    Similarly to the argument in the proof of   
    \cite[Theorem~2.3.12]{Epsteinbook},
    it can be seen that 
    the loop $w$ can be subdivided into 
    at most $K_0n^2$ loops of length at 
    most $\ell(n) = 4h'(K_0n) + K_1$ for some 
    integer constants $K_0$ and $K_1$.
    Therefore, $D(n) \leqslant  K_0 n^2 D(\ell(n))$ 
    which implies that $D(n) \preceq n^2 D(\ell(n))$.  
  
    For the group $BS(p,q)$ the Dehn function 
    is at most exponential 
    (see \cite[\S~7.4]{Epsteinbook}), i.e., 
    $D(n) \leqslant \lambda^n$ for some 
    constant $\lambda$. Therefore, 
    $D(n) \preceq n^2 \lambda^{\ell(n)}$.
    Clearly, $\ell \preceq h'$ which implies that
    $\ell \prec \mathfrak{i}$. 
    Let us show that  
    $n^2 \lambda ^{\ell(n)} \prec \mathfrak{e}$. 
    It can be seen that 
    $n^2 \lambda ^{\ell(n)} \preceq \mathfrak{e}$. 
    Assume that 
    $\mathfrak{e} \preceq n^2 \lambda^{\ell(n)}$. 
    Then, for all sufficiently large $n$
    and some constants $K$ and $M$ we have: 
    $\exp(n) \leqslant K n^2 \lambda^{\ell(Mn)}$.
    This implies that 
    $n - 2 \ln n - \ln K \leqslant 
    (\ln \lambda) \ell (Mn)$. Clearly, 
    $\frac{n}{2} \leqslant n - 2 \ln n - \ln K$
    for all sufficiently large $n$, and, therefore,  
    $n \leqslant (2\ln \lambda ) \ell (Mn)$.
    This implies that $\mathfrak{i} \preceq \ell$
    which contradicts to the inequality 
    $\ell \prec \mathfrak{i}$. Thus, 
    $D(n)\preceq n^2 \lambda ^{\ell(n)} 
    \prec \mathfrak{e}$ which implies that 
    $D(n) \prec \mathfrak{e}$. The last inequality 
    contradicts to the fact that for the group 
    $BS(p,q)$ the Dehn function is at least 
    exponential, i.e.,  
    $D(n) \geqslant \mu ^n$ for some constant $\mu$
    (see \cite[\S~7.4]{Epsteinbook})    
    which implies that $\mathfrak{e} \preceq D(n)$.  \qed       
 \end{proof}

   \section{The Lamplighter Group} 
 \label{seclamplighter}  
   
   The lamplighter group is 
   the wreath product  $\mathbb{Z}_2 \wr \mathbb{Z}$
   of the cyclic group 
   $\mathbb{Z}_2$ and 
   the infinite cyclic group 
   $\mathbb{Z}$. For the definition of 
   the wreath product of groups we refer the 
   reader to \cite{KargapolovMerzljakov}.           
   Let $t$ be a generator of the cyclic group 
   $\mathbb{Z}= \langle t \rangle$ and  
   $a$ be the nontirival element of the group 
   $\mathbb{Z}_2$. The canonical 
   embeddings of the groups $\mathbb{Z}_2$ and 
   $\mathbb{Z}$ into 
   the wreath product $\mathbb{Z}_2 \wr \mathbb{Z}$ 
   enable us to consider 
   $\mathbb{Z}_2$ and $\mathbb{Z}$ as the 
   subgroups of $\mathbb{Z}_2 \wr \mathbb{Z}$. 
   With respect to the 
   generators $a$ and $t$, the lamplighter  
   group  has the presentation 
   $  \langle a, t 
      \, | \,  
      [t^i a t^{-i}, t^{j} a t^{-j}], a^2 \rangle 
   $.    
   The lamplighter group is not finitely presented 
   \cite{BS62}, and, therefore, 
   it is not automatic 
   due to     
   \cite[Theorem~2.3.12]{Epsteinbook}.     
   
   The elements of the lamplighter group have 
   the following geometric
   interpretation. 
   Every element of the lamplighter group corresponds to 
   a bi--infinite string of lamps, indexed by 
   integers $i \in \mathbb{Z}$, 
   each of which
   is either lit or unlit, such that 
   only finite number of lamps are lit, and the 
   lamplighter pointing at the current lamp 
   $i = m$. 
   The identity of the lamplighter group    
   corresponds to the configuration when all lamps 
   are unlit and the lamplighter points at the lamp
   positioned at the origin $m = 0$. 
   
   The right multiplication by $a$ changes 
   the state of the current lamp. 
   The right multiplication by $t$ (or $t^{-1}$) 
   moves the lamplighter to the right 
   $m \mapsto  m + 1$
   (or to the left $m \mapsto m - 1$).     
   The elements of the subgroup 
   $\mathbb{Z} \leqslant 
   \mathbb{Z}_2 \wr \mathbb{Z}$ are 
   the configurations for which all lamps are unlit. 
   For the elements of the subgroup 
   $\mathbb{Z}_2 \leqslant \mathbb{Z}_2 \wr \mathbb{Z}$ 
   all lamps, apart from 
   the one at the origin, are unlit and the 
   lamplighter points at the lamp 
   positioned at the origin, which 
   can be either lit or unlit. 
           
   For any given integer $i \in \mathbb{Z}$ 
   we put $a_i = t^i a t^{-i}$. 
   The group element $a_i$ corresponds to 
   the configuration when the lamp at 
   the position $i$ is lit, 
   all other lamps are unlit and 
   the lamplighter points at the origin $m=0$.
   Let $g$ be an element of the lamplighter group. 
   The 'right--first' and 
   the 'left--first' normal forms of $g$ are defined 
   as follows:
   \begin{equation*}
      rf(g) = a_{i_1} a_{i_2} \dots a_{i_k} 
              a_{-j_1} a_{-j_2} \dots a_{-j_l} t^m, 
   \end{equation*}
   \begin{equation*}
     lf (g) = a_{-j_1} a_{-j_2} \dots a_{-j_l} 
              a_{i_1}  a_{i_2} \dots a_{i_k} t^m,   
   \end{equation*}   
   where $i_k > \dots > i_2 > i_1 \geqslant 0$, 
   $j_l > \dots > j_1 > 0$ and the lamplighter 
   points at the position $m$ (see \cite{Taback03}). 
   For the element 
   $g$ the lit lamps are at the  
   positions $-j_l, \dots , - j_1, i_1,\dots, i_k$ 
   and the lamplighter points at the position $m$.
   In 'right--first' normal form the lamplighter 
   moves to the right illuminating the appropriate 
   lamps until it reaches the lamp at the position $i_k$. 
   Then it moves back to the origin, and then further 
   to the left illuminating the appropriate lamps 
   until it reaches the lamp at the position $-j_l$. 
   After that the lamplighter moves to the position $m$.   
   Let $A=\{a,t\}$ and $S=\{a,a^{-1},t, t^{-1}\}$.
   \begin{proposition}
   \label{L2distprop1}    
   (\cite[Proposition~3.2]{Taback03})
      The word length of the element $g$ with 
      respect to the generating set $A$ 
      is given by 
      \begin{equation*} 
         d_A (g) = k + l + 
         \min\{2 i_k + j_l + |m+j_l|, 
               2 j_l + i_k + |m-i_k|\}.    
      \end{equation*} 
   \end{proposition}
   Some Cayley automatic 
   representations of $\mathbb{Z}_2 \wr \mathbb{Z}$
    had been obtained in 
   \cite{KKM11,dlt14,berdkhouss15}. Let us now 
   construct a new Cayley automatic representation 
   of  $\mathbb{Z}_2 \wr \mathbb{Z}$ 
   using the 'right--first' normal form 
   which is compatible with Definition
   \ref{defcayleyautomatic}. 
   For a given element $g$ of the lamplighter group
   we construct the word $w = u'v'$ which is the 
   concatenation of two words $u',v' \in S^*$. 
   The words $u'$ and $v'$ are obtained from 
   the words $u$ and $v$, defined below, 
   by canceling adjacent opposite powers of $t$.    
   Assume first that $m \geqslant 0$.     
   \begin{itemize}   
   \item{Suppose that
         $\{i_1, \dots , i_k \} = \varnothing$ or         
         $\{i_1, \dots , i_k \} \neq \varnothing$ 
         and $m>i_k$. 
         We put 
         $u = t^{i_1}at^{-i_1} \dots 
         t^{i_k} a t^{-i_k} t^m aa$.
         We put 
         $v = t^{-j_1}a t^{j_1} \dots t^{-j_l} a$.}
   \item{Suppose that 
         $\{i_1,\dots, i_k\} \neq \varnothing$ 
         and $m \leqslant i_k$. If $m=i_n$ for
         some $n=1,\dots,k$, then we put  
         $u = t^{i_1}at^{-i_1} \dots 
         t^{i_n} aaa t^{-i_n} \dots t^{i_k} a$.
         Otherwise, either $m< i_{1}$ or
         there exists $q = 1,\dots,k-1$ for 
         which $i_{q} < m < i_{q+1}$.           
         In the first case we put 
         $u = t^{m} aa t^{-m} t^{i_1} a t^{-i_1} 
         \dots t^{i_k} a$.          
         In the latter case we put
         $u = t^{i_1}a t^{-i_1} \dots  
          t^{i_{q}} a t^{-i_q} t^{m} aa t^{-m} 
          t^{i_{q+1}} a t^{-i_{q+1}} \dots 
          t^{i_k}a$.    
          The word $v$ is the same 
          as above. 
         }
    \end{itemize}    
    Assume now that $m<0$. 
    \begin{itemize}   
    \item{Suppose that 
          $\{ j_1, \dots, j_l \} = \varnothing$ or  
          $\{j_1, \dots, j_l \} \neq \varnothing$ 
          and $m<-j_l$. 
          We put 
          $v = t^{-j_1}at^{j_1} \dots 
           t^{-j_l}at^{j_l}t^maa$. 
          We put 
          $u = t^{i_1} a t^{-i_1} \dots t^{i_k} a$.                  
          } 
    \item{Suppose that  
          $\{j_1, \dots, j_l \} \neq \varnothing$ 
          and $m \geqslant - j_l$. If $m=-j_n$ 
          for some $n=1,\dots,l$, then we put 
          $v = t^{-j_1}a t^{j_1} \dots 
           t^{-j_n}aaat^{j_n} \dots t^{-j_l} a$. 
          Otherwise, either $m> -j_1$ or there 
          exists $q = 1, \dots, l-1$ for which 
          $-j_q > m > -j_{q+1}$. In the first case 
          we put $v = t^{m} aa t^{-m} 
          t^{-j_1} a t^{j_1} \dots 
          t^{-j_l} a t^{j_l}$. In the latter 
          case we put 
          $v = t^{-j_1} a t^{j_1} \dots 
          t^{-j_q} a t^{j_q} t^m aa t^{-m} 
          t^{-j_{q+1}} a t^{j_{q+1}} 
          \dots t^{-j_l} a$. The word $u$ is 
          the same as above.       
          }  
    \end{itemize}       
    Let us show two simple examples. Suppose 
    first that
    the lit lamps are at the positions 
    $-1,0,2$ and the lamplighter is at the 
    position $m=1$. Then, for the corresponding group 
    element, the word $w$ is $ataatat^{-1}a$. Suppose 
    now that the lit lamps are at the positions 
    $-1,1$ and the lamplighter is at the position 
    $m=-1$. Then, for the corresponding group element, 
    the word $w$ is $tat^{-1}aaa$.           
    The set of all such words $w$ forms some 
    language $L \subseteq S^*$.     
    Thus, we have constructed the bijection
    $\psi : L \rightarrow \mathbb{Z}_2 \wr 
    \mathbb{Z}$.  
    
    It can be verified that $L$ is a regular 
    language and $\psi$ provides a Cayley 
    automatic representation of the lamplighter 
    group in 
    the sense of Definition \ref{defcayleyautomatic}. 
       We note that in the Cayley automatic 
       representation 
       $\psi: L \rightarrow \mathbb{Z}_2 \wr \mathbb{Z}$        
       constructed above we use the subwords 
       $aa$ and $aaa$ to specify the 
       lamplighter position. We use $aa$ 
       and $aaa$ if the lamp, the lamplighter 
       is pointing at, is unlit and lit, 
       respectively. It is worth noting that if
       $g \in \mathbb{Z}_2 \wr \mathbb{Z}$ 
       is an element 
       for which all lamps at negative positions
       $j<0$ are unlit and $m \geqslant i_k$, then
       for $w = \psi^{-1}(g)$ we obtain 
       that $\pi (w)= \psi (w)$. That is, 
       on a certain infinite subset of
       $L$ the maps $\pi$ and $\psi$ coincide.                       
   \begin{theorem} 
   \label{L2theorem1}      
      The lamplighter group 
      $\mathbb{Z}_2 \wr \mathbb{Z} \in 
      \mathcal{B}_\mathfrak{i}$. 
      Moreover, for any $f \prec \mathfrak{i}$, 
      $\mathbb{Z}_2 \wr \mathbb{Z} 
       \notin \mathcal{B}_f$.   
   \end{theorem}
   \begin{proof} 
     Let us consider the Cayley automatic representation 
     $\psi: L \rightarrow \mathbb{Z}_2 \wr \mathbb{Z}$  
     constructed above. 
     Let $h$ be the function given by
     \eqref{maineq1} with respect to the 
     Cayley automatic representation $\psi$. 
     We will show that $h \preceq \mathfrak{i}$
     (in fact one can verify  that 
     $h \asymp \mathfrak{i}$).
     For a given $n$ let $w \in L ^{\leqslant n}$ 
     be a word and $g=\psi(w)$ 
     be the 
     corresponding group element of 
     $\mathbb{Z}_2 \wr \mathbb{Z}$. 
     Clearly, we have that 
     $d_A (\pi(w),\psi(w)) \leqslant n + d_A (g)$. 
     Therefore, it suffices to show that 
     $d_A (g) \leqslant Cn$ for some constant $C$. 
    
     It follows from the construction of 
     $w=\psi^{-1}(g)$ that if $m\geqslant 0$, then 
     $|w| = k + l + \max \{m, i_k \} + j_l + 2$, 
     and if $m<0$, then 
     $|w|=k+l + \max \{-m,j_l\} + i_k + 2$.    
     By Proposition \ref{L2distprop1}, we obtain that 
     $d_A (g) \leqslant 3 |w| \leqslant 3 n$.  
     Therefore, $h \preceq \mathfrak{i}$
     which implies that    
     $\mathbb{Z}_2 \wr \mathbb{Z} \in 
     \mathcal{B}_\mathfrak{i}$.          
     Let us show the second statement of the 
     theorem. 
     For a given $m>0$, let 
     $R_m$ be the following set of relations
        $R_m=\{a^2\} \cup   
           \{[t^i a t^{-i}, t^j a t^{-j}]\,|\, 
             -m \leqslant i < j \leqslant m\}$. 
     We first notice  that
     for any loop 
     $w \in S^*,|w| \leqslant l$     
     in the lamplighter group 
     $\mathbb{Z}_2 \wr \mathbb{Z}$
     the word $w$ can be represented as a product 
     of conjugates of the relations from $R_l$, i.e.,      
     the identity 
     $w =  \prod_{i=1}^{k} 
      v_i r_i ^{\pm 1} v_i^{-1}$ 
     holds in the free 
     group $F(A)$
     for some $v_i \in S^*$ and $r_i \in R_l, 
     i =1, \dots, k$.     
   
     Suppose now that 
     $\mathbb{Z}_2 \wr \mathbb{Z} \in \mathcal{B}_f$ 
     for some $f \prec \mathfrak{i}$.                
     Similarly to Theorem \ref{BStheorem1}, we 
     obtain that       
     then there exists a function  
     $\ell \prec \mathfrak{i}$ such that
     any loop  $w$ of the length 
     less than or equal to $n$
     can be subdivided into loops of the length 
     at most $\ell(n)$. Therefore, 
     for any loop given by a word 
     $w \in S^*, |w| \leqslant n$, 
     the identity 
     $w =  \prod_{i=1}^{k} 
      v_i r_i ^{\pm 1} v_i^{-1}$     
     holds in the free group $F(A)$     
     for some $v_i \in S^*$ and 
     $r_i \in R_{\ell(n)}$, $i=1,\dots,k$.      
     In particular, every 
     relation from $R_{n}$ can be expressed
     as a product of conjugates of the 
     relations from $R_{\ell(8n+4)}$ (the longest 
     relation from $R_{n}$ is  
     $[t^{-n}at^{n},t^{n}at^{-n}]$ which has the
     length $8n+4$).     
     However,  
     not every relation from $R_{n}$ can be expressed 
     as a product of conjugates of the relations 
     from $R_{n-1} \subset R_{n}$ because 
     the groups 
     $G_n = \langle a,t | R_n \rangle$ and 
     $G_{n-1} = \langle a, t | R_{n-1} \rangle$ 
     are not isomorphic.         
     This implies the inequality 
     $\ell(8n+4) \geqslant n$ leading to a contradiction 
     with $\ell \prec \mathfrak{i}$. 
     The fact that $G_n = \langle a,t | R_n \rangle$ and 
     $G_{n-1} = \langle a, t | R_{n-1} \rangle$ are 
     not isomorphic can be shown as follows. 
     
     The group
     $G_n$ can be represented as 
     $G_n = 
      \langle a_{-n},\dots,a_0,\dots,a_n  | 
      a_0 ^2; a_{i-1} = t^{-1} a_i t, 
      i = - (n-1),\dots,n; [a_i,a_j], 
      i,j = -n, \dots,n \rangle$, 
      so $G_n$ is the HNN extension 
      of the base group
      $\bigoplus_{i=-n}^{n} \mathbb{Z}_2
       = \langle a_{-n}, \dots, a_n | 
          a_i ^2 , [a_i, a_j] \rangle $ 
      relative to the isomorphism 
      $\varphi_n$ between the 
      subgroups $A_n,B_n \leqslant G_n$ 
      generated by $a_{-(n-1)},\dots,a_n$ 
      and $a_{-n},\dots,a_{n-1}$, respectively, 
      for which $\varphi_n: a_i \mapsto a_{i-1},
      i = -(n-1),\dots,n$.        
      As a consequence of 
      Britton's lemma \cite{LyndonSchuppbook}, 
      we have the property 
      that every finite subgroup of an HNN extension 
      is conjugate to a finite subgroup of its 
      base group.      
      Assuming that $G_{n+1}$ and 
      $G_n$ are isomorphic, we obtain 
      that $\bigoplus_{i=-(n+1)}^{n+1} 
      \mathbb{Z}_2$
      can be embedded into 
      $\bigoplus_{i=-n}^{n} \mathbb{Z}_2$ 
      which leads to a contradiction. \qed            
   \end{proof}

 \section{The Heisenberg Group}  
 \label{secheisenberggroup}  
   The Heisenberg group $\mathcal{H}_3 (\mathbb{Z})$ 
   is the group of all matrices of the form:    
   \begin{equation*}
   \mathcal{}\left(
   \begin{array}{ccc}
    1 & x & z \\
    0 & 1 & y \\
    0 & 0 & 1
   \end{array}
   \right),
   \end{equation*}
   where $x,y$ and $z$ are integers. 
   Every element $g \in \mathcal{H}_3(\mathbb{Z})$ 
   corresponds to a triple $(x,y,z)$. 
   Let  $s$ be a group element of 
   $\mathcal{H}_3$ corresponding 
   to the triple $(1,0,0)$, 
   $p$ corresponding to $(0,1,0)$, and 
   $q$ corresponding to $(0,0,1)$. 
   If $g$ corresponds to a triple $(x,y,z)$, 
   then $gs,gp$ and $gq$ 
   correspond to the triples 
   $(x+1,y,z)$, $(x,y+1,x+z)$ and $(x,y,z+1)$, 
   respectively.    
   The observation that $\mathcal{H}_3$ is 
   not an automatic group 
   but its Cayley graph is automatic was first made 
   by S\'{e}nizergues.   
    
   
   The Heisenberg group $\mathcal{H}_3$
   is isomorphic to the group 
   $\langle s,p,q | s^{-1}p^{-1}sp=q, sq=qs,  
   pq=qp \rangle$, and 
   it can be generated by the elements $s$ and $p$. 
   The exact distance formula on 
   $\mathcal{H}_3(\mathbb{Z})$ for the  
   generating set $\{s,p\}$ is obtained in 
   \cite[Theorem~2.2]{Blachere03}. 
   However, for our purposes it is enough to 
   have the metric estimates 
   which the reader can find in 
   \cite[Proposition~1.38]{Roebook}.   
   Let $A= \{e,s,p,q\}$ and 
   $S=A \cup A^{-1} = \{e,s,p,q,s^{-1},p^{-1},q^{-1}\}$.       
   \begin{proposition}(\cite[Proposition~1.38]{Roebook})
   \label{H3metricestprop1}   
   There exist constants $C_1$ and $C_2$ such that 
   for an element $g \in \mathcal{H}_3$ corresponding
   to a triple $(x,y,z)$ we have 
   \begin{equation*}
      C_1 (|x| + |y| + \sqrt{|z|}) 
      \leqslant d_A (g) \leqslant  
      C_2 (|x| + |y| + \sqrt{|z|}).
   \end{equation*} 
   \end{proposition}    
   \begin{proof} 
      We first get an upper bound. 
      Every group element $g \in \mathcal{H}_3$ 
      can be represented as $s^n p^m q^l$ 
      corresponding to the triple $(x,y,z)=(n,m,nm+l)$. 
      It can be verified that 
      $s^k p^k s^{-k} p^{-k} = q^{k^2}$. Therefore, 
      the length of $q^l$ is at most 
      $6 \sqrt{|l|} \leqslant 
      6 \sqrt{|z|} + 3|n|+ 3|m|$. 
      For $C_2=6$ we obtain the required upper bound. 
      Let us prove now a lower bound. 
      If $d_A(g)=r$ for an element $g$ corresponding 
      to a triple $(x,y,z)$, 
      then $|x|,|y|\leqslant r$ and 
      $|z| \leqslant r + r^2$. For $C_1= \frac{1}{4}$ 
      we obtain the required lower bound.    \qed     
   \end{proof}

   Let us construct a Cayley automatic representation 
   of the Heisenberg group $\mathcal{H}_3$ which 
   is compatible with Definition 
   \ref{defcayleyautomatic}. 
   For a given $g \in \mathcal{H}_3$ 
   corresponding to a triple $(x,y,z)$ 
   we construct the word $w=u v$ which 
   is the concatenation of two words $u,v \in S^*$
   constructed as follows. 
   We put $u = p^y$. 
   Let $b_x$ and $b_z$ be the binary representations  
   of the integers $|x|$ and $|z|$ 
   (with the least significant digits first).  
   We put $b$ to be $b_x \otimes b_z$ with the 
   padding symbol $\diamond$ changed to $0$.   
   The word $b$ is a word  over the alphabet  
   consisting of the symbols    
   {\tiny $\left(\begin{array}{c} 0 \\ 0 \end{array} 
   \right), \left(\begin{array}{c} 0 \\ 1 \end{array}
   \right),
    \left(\begin{array}{c} 1 \\ 0 \end{array}\right),
    \left(\begin{array}{c} 1 \\ 1 \end{array}\right)$}. 
   
   Replacing the symbols 
   {\tiny $\left(\begin{array}{c} 0 \\ 0 \end{array} 
    \right),
    \left(\begin{array}{c} 0 \\ 1 \end{array}\right),
    \left(\begin{array}{c} 1 \\ 0 \end{array}\right),
    \left(\begin{array}{c} 1 \\ 1 \end{array}\right)$}
    in  $b$ by the words $ee$, $eq$, $se$ and $sq$ 
    we obtain 
    a word $b' \in \{e,s,q\}^*$. 
    If $x \geqslant 0$ and $z \geqslant 0$, 
    then we put $v = b'$.     
    If $x <0$ or 
    $z < 0$, then $v$ is obtained from $b'$ 
    by replacing the symbols $s$ and $q$ to the symbols
    $s^{-1}$ and $q^{-1}$, respectively.  
    For example, the triple $(3,-3,-4)$ is represented 
    by the word $p^{-1}p^{-1}p^{-1}seseeq^{-1}$.
    The set of all such words $w$ is a regular 
    language $L \subseteq S^*$. Thus,
    we have constructed the bijection 
    $\psi: L \rightarrow \mathcal{H}_3$.   
    It can be verified that $\psi$ provides a 
    Cayley automatic representation of   
    the Heisenberg group $\mathcal{H}_3$. 

    It is worth noting that if $g \in \mathcal{H}_3$ 
    corresponds to a triple $(0,y,0)$, then 
    for the word $w = \psi^{-1}(g)$ we have 
    $\psi(w) = \pi(w)$. That is, the
    maps $\pi$ and $\psi$ coincide if 
    restricted on the cyclic 
    subgroup $\langle p\rangle \leqslant \mathcal{H}_3$.
    \begin{theorem} 
    \label{H3thm1}       
      The Heisenberg group       
      $\mathcal{H}_3 \in \mathcal{B}_\mathfrak{e}$.
      Moreover, for any $f \prec \sqrt[3]{n}$,
      $\mathcal{H}_3 \notin \mathcal{B}_f$.     
    \end{theorem}     
    \begin{proof}         
       Let $h$ be the function given 
       by \eqref{maineq1} with respect 
       to the Cayley automatic representation
       $\psi: L \rightarrow 
        \mathcal{H}_3$ constructed above.       
       We will show that $h \asymp \mathfrak{e}$. 
       Although for the first statement of  
       the theorem it is enough to show that  
       $h \preceq \mathfrak{e}$, the inequality 
       $\mathfrak{e} \preceq h$ guarantees that we 
       cannot get a better result using just 
       the representation $\psi$.             
       Let $w=uv \in L ^{\leqslant n}$ and  
       $g=\psi(w)$ be 
       the group element of  $\mathcal{H}_3$
       corresponding to a triple $(x,y,z)$.               
       By Proposition \ref{H3metricestprop1},  
       there exists a constant $C_2$ such that 
       $d_A(g) \leqslant C_2(|x|+|y|+\sqrt{|z|})
        \leqslant C_2 (2^{|v|}+|u|+\sqrt{2^{|v|}}) 
        \leqslant 2C_2 2^{|u|+|v|} 
        \leqslant 2C_2 \exp{(|w|)} \leqslant 
        2 C_2 \exp{(n)}$. 
      Therefore,  
      $h \preceq \mathfrak{e}$ which implies that  
      $\mathcal{H}_3 \in \mathcal{B}_\mathfrak{e}$.
      
      Let us show now that $\mathfrak{e} \preceq h$.  
      Let $g_i=s^i,i \geqslant 2$. The length 
      of the corresponding word $w_i= \psi^{-1}(g_i)$ 
      is equal to the doubled length of the binary 
      representation of $i$. We have 
      $d_A(\pi(w_i),\psi(w_i))=
      d_A (\pi(w_i)^{-1}s^i)= d_A (s^{n_i})$ 
      for some positive integer $n_i$. Clearly, 
      there exists a constant $C$ such that 
      $n_i \geqslant C i$. The group element $s^{n_i}$
      corresponds to the triple $(n_i,0,0)$. 
      By Proposition \ref{H3metricestprop1}, 
      we have $d_A (s^{n_i}) \geqslant C_1 n_i$. 
      Therefore, there exists a constant 
      $C'>0$ such that 
      $d_A(s^{n_i}) \geqslant C' 2^{\frac{|w_i|}{2}}$ 
      for all $i \geqslant 2$. This implies 
      that $\mathfrak{e} \preceq h$. Therefore, 
      $h \asymp \mathfrak{e}$.    
      
      Let us show now the second statement of the 
      theorem. Repeating exactly the same argument 
      as used in Theorem \ref{BStheorem1}, we conclude 
      that there exists a function 
      $\ell(n) \prec \sqrt[3]{n}$ for which the 
      inequality $D(n)\preceq n^2 D(\ell(n))$ 
      holds, where $D(n)$ is the Dehn function 
      of $\mathcal{H}_3$. For the group 
      $\mathcal{H}_3$ the Dehn function is at most 
      cubic; specifically for the presentation 
      $\mathcal{H}_3 = \langle s,p,q | 
      s^{-1}p^{-1}sp=q,sq=qs,pq=qp \rangle$,
      $D(n)\leqslant n^3$
      (see \cite[\S~8.1]{Epsteinbook}). 
      Therefore, $D(n) \preceq n^2 \ell(n)^3$. 
      Let us show that 
      $n^2 \ell(n)^3 \prec n^3$. It can be seen 
      that $n^2 \ell(n)^3  \preceq n^3$. Assume
      that $n^3 \preceq n^2 \ell(n)^3$. Then, 
      for all sufficiently large $n$ and some 
      constants $K$ and $M$ we have: 
      $n^3 \leqslant K n^2 \ell (Mn)^3$. 
      This implies that 
      $\sqrt[3]{n} \leqslant 
       \sqrt[3]{K} \ell(Mn)$. 
      Therefore, $\sqrt[3]{n} \preceq \ell(n)$ which 
      contradicts to the inequality 
      $\ell(n) \prec \sqrt[3]{n}$. Thus, 
      $D(n) \preceq n^2 \ell(n)^3 \prec n^3$ which   
      implies that $D(n) \prec n^3$. 
      The last inequality contradicts to the fact
      that the Dehn function is at least 
      cubic (see \cite[\S~8.1]{Epsteinbook}) which
      implies that $n^3 \preceq D(n)$.   \qed         
    \end{proof}

    \section{Discussion} 
    \label{secdiscussion}    
    
   In this paper we proposed a way  to 
   measure closeness of Cayley automatic groups 
   to the class of automatic groups. 
   We did this by introducing the classes of 
   Cayley automatic groups $\mathcal{B}_f$ 
   for the functions $f \in \mathfrak{F}$. 
   In Theorem \ref{chartheorem1} we 
   characterized non--automatic groups 
   by showing that for any such group $G$ in
   some class $\mathcal{B}_f$ the function 
   $f$ must be unbounded. We studied 
   then the cases of the Baumslag--Solitar 
   groups $BS(p,q), 1 \leqslant p < q$, the 
   lamplighter group and the Heisenberg 
   group $\mathcal{H}_3$. In Theorems 
   \ref{BStheorem1} and \ref{L2theorem1} 
   we proved that the Baumslag--Solitar groups and 
   the lamplighter group are in the class 
   $\mathcal{B}_\mathfrak{i}$ and they 
   cannot belong to any class $\mathcal{B}_f$ 
   for which $f \prec \mathfrak{i}$. 
   For the Heisenberg group $\mathcal{H}_3$ 
   in Theorem \ref{H3thm1}   
   we proved that  
   $\mathcal{H}_3 \in \mathcal{B}_\mathfrak{e}$, 
   but we could only prove that it cannot 
   belong to any class $\mathcal{B}_f$ for which 
   $f \prec \sqrt[3]{n}$. The following 
   questions are apparent from the results 
   obtained in this paper. 
   \begin{itemize}   
   \item{Is there any unbounded function 
   $f \prec \mathfrak{i}$ for 
   which the class $\mathcal{B}_f$ contains 
   a non--automatic group?} 
   \item{Is there any function 
   $f\prec \mathfrak{e}$ for which 
   $\mathcal{H}_3 \in \mathcal{B}_f$?}
   \item{Is there any characterization of 
         a class $\mathcal{B}_f$,
         where $f$ is an unbounded function?}   
   \end{itemize}

 \section*{Acknowledgments} The authors thank the 
 referees for useful comments. 

 \bibliographystyle{splncs03}

\bibliography{measuringcloseness}

\end{document}